\documentclass[12pt,a4paper]{amsart}
\usepackage[utf8]{inputenc}
\usepackage[english]{babel}
\usepackage{graphicx}
\usepackage{amsmath}
\usepackage{amssymb}
\usepackage{amscd}
\usepackage{amsthm}
\usepackage{amscd}

\usepackage[matrix,arrow,curve]{xy}
\usepackage{stmaryrd}
\usepackage{hyperref}

\usepackage{tikz}
\usetikzlibrary{arrows.meta}
\usetikzlibrary{bending}

\usepackage{graphicx}
\usepackage{xspace}
\usepackage{enumerate}
\usepackage{verbatim}
\usepackage{geometry}
\usepackage{tikz-cd}

\usepackage{mathrsfs}

\DeclareMathOperator{\K}{\mathrm{K}}
\DeclareMathOperator{\F}{\mathsf{F}}

\begin{document}

\newcommand{\Sm}{\mathcal{S}^{\!}\mathsf{m}_k}
\newcommand{\Rings}{\mathcal{R}^{\!}\mathsf{ings}^*}
\newcommand{\rings}{\mathcal{R}^{\!}\mathsf{ings}}
\newcommand{\SmE}[1]{\mathcal{S}^{\!}\mathsf{m}_{#1}}
\newcommand{\Mot}[1]{\mathcal M^{\!}\mathsf{ot}_{#1}}
\newcommand{\Corr}[1]{\mathcal C^{\!}\mathsf{orr}_{#1}}
\newcommand{\MotF}[1]{\mathcal M_{#1}}
\newcommand{\QG}{\mathbb Q\Gamma}
\newcommand{\Inv}{\mathrm{Inv}(\QG)}
\newcommand{\Gal}{\mathrm{Gal}(k^{\mathrm{sep}}/k)}
\newcommand{\End}{\mathrm{End}}
\newcommand{\M}[1]{\mathcal{M}_{#1}}
\newcommand{\EG}{\!\,_EG}
\newcommand{\EGP}{\!\,_E(G/P)}
\newcommand{\EX}{\!\,_EX}
\newcommand{\XG}{\!\,_{\xi}G}
\newcommand{\XGP}{\!\,_{\xi}(G/P)}
\newcommand{\KQ}[1]{\mathrm K(n)^*\big(#1;\,\mathbb Q[v_n^{\pm1}]\big)}
\newcommand{\KZ}[1]{\mathrm K(n)^*\big(#1;\,\mathbb Z_{(p)}[v_n^{\pm1}]\big)}
\newcommand{\KZp}[1]{\mathrm K(n)^*\big(#1;\,\mathbb Z_p[v_n^{\pm1}]\big)}
\newcommand{\KF}[1]{\mathrm K(n)^*\big(#1;\,\mathbb F_p[v_n^{\pm1}]\big)}
\newcommand{\CHQ}[1]{\mathrm{CH}^*\big(#1;\,\mathbb Q[v_n^{\pm1}]\big)}
\newcommand{\KXZ}[1]{\!\,^{\mathrm K(n)\!}#1_{\,\mathbb Z_{(p)}[v_n^{\pm1}]}}
\newcommand{\KMotQ}{\Mot{\,\mathrm K(n)}}
\newcommand{\CHMotQv}{\Mot{\,\mathrm{CH}}}
\newcommand{\KXQ}[1]{\mathcal M_{\mathrm K(n)}(#1)}
\newcommand{\KMQ}{\mathcal M_{\,\mathrm K(n)}}
\newcommand{\CHMQv}{\MotF{\,\mathrm{CH}}}
\newcommand{\KCorrQ}{\Corr{\,\mathrm K(n)}}
\newcommand{\CHCorrQv}{\Corr{\,\mathrm{CH}}}
\newcommand{\CHCorrQ}{\Corr{\,\mathrm{CH}}}
\newcommand{\AMot}{\Mot A}

\newcommand{\e}{\varepsilon}
\newcommand{\con}{\ensuremath{\triangledown}}
\newcommand{\ra}{\ensuremath{\rightarrow}}
\newcommand{\tp}{\ensuremath{\otimes}}
\newcommand{\pr}{\ensuremath{\partial}}
\newcommand{\trigd}{\ensuremath{\triangledown}}
\newcommand{\dAB}{\ensuremath{\Omega_{A/B}}}
\newcommand{\QQ}{\ensuremath{\mathbb{Q}}\xspace}
\newcommand{\CC}{\ensuremath{\mathbb{C}}\xspace}
\newcommand{\RR}{\ensuremath{\mathbb{R}}\xspace}
\newcommand{\ZZ}{\ensuremath{\mathbb{Z}}\xspace}
\newcommand{\Zp}{\ensuremath{\mathbb{Z}_{(p)}}\xspace}
\newcommand{\Z}[1]{\ensuremath{\mathbb{Z}_{(#1)}}\xspace}
\newcommand{\NN}{\ensuremath{\mathbb{N}}\xspace}
\newcommand{\LL}{\ensuremath{\mathbb{L}}\xspace}
\newcommand{\inN}{\ensuremath{\in\mathbb{N}}\xspace}
\newcommand{\inQ}{\ensuremath{\in\mathbb{Q}}\xspace}
\newcommand{\inR}{\ensuremath{\in\mathbb{R}}\xspace}
\newcommand{\inC}{\ensuremath{\in\mathbb{C}}\xspace}
\newcommand{\OO}{\ensuremath{\mathcal{O}}\xspace}
\newcommand{\rarr}{\rightarrow}
\newcommand{\Rarr}{\Rightarrow}
\newcommand{\xrarr}[1]{\xrightarrow{#1}}
\newcommand{\larr}{\leftarrow}
\newcommand{\lrarr}{\leftrightarrows}
\newcommand{\rlarr}{\rightleftarrows}
\newcommand{\rrarr}{\rightrightarrows}
\newcommand{\al}{\alpha}
\newcommand{\bt}{\beta}
\newcommand{\ld}{\lambda}
\newcommand{\om}{\omega}
\newcommand{\Kd}[1]{\ensuremath{\Omega^{#1}}}
\newcommand{\KKd}{\ensuremath{\Omega^2}}
\newcommand{\vd}{\partial}
\newcommand{\PC}{\ensuremath{\mathbb{P}_1(\mathbb{C})}}
\newcommand{\PPC}{\ensuremath{\mathbb{P}_2(\mathbb{C})}}
\newcommand{\derz}{\ensuremath{\frac{\partial}{\partial z}}}
\newcommand{\derw}{\ensuremath{\frac{\partial}{\partial w}}}
\newcommand{\mb}[1]{\ensuremath{\mathbb{#1}}}
\newcommand{\mf}[1]{\ensuremath{\mathfrak{#1}}}
\newcommand{\mc}[1]{\ensuremath{\mathcal{#1}}}
\newcommand{\id}{\ensuremath{\mbox{id}}}
\newcommand{\dd}{\ensuremath{\delta}}
\newcommand{\bu}{\bullet}
\newcommand{\ot}{\otimes}
\newcommand{\boxt}{\boxtimes}
\newcommand{\op}{\oplus}
\newcommand{\mt}{\times}
\newcommand{\Gm}{\mathbb{G}_m}
\newcommand{\Ext}{\ensuremath{\mathrm{Ext}}}
\newcommand{\Tor}{\ensuremath{\mathrm{Tor}}}

\newcommand{\kn}[1]{\mathrm K(n)^*(#1)}
\newcommand{\ckn}[1]{\mathrm{CK}(n)^*(#1)}
\newcommand{\grckn}[1]{\mathrm{gr}_\tau^{*}\,\mathrm{CK}(n)^{*}(#1)}
\newcommand{\so}{\mathrm{SO}_m}
\newcommand{\pt}{\mathrm{pt}}
\newcommand{\sic}{\mathrm{sc}}
\newcommand{\ad}{\mathrm{ad}}
\newcommand{\St}{\mathrm{St}}
\newcommand{\SL}{\mathrm{SL}}
\newcommand{\Sp}{\mathrm{Sp}}
\newcommand{\Spin}{\mathrm{Spin}}

\newcommand{\A}{\mathsf{A}}
\newcommand{\C}{\mathsf{C}}

\newcommand{\GF}{\mathbb{F}}

\makeatletter
\newcommand{\colim@}[2]{%
  \vtop{\m@th\ialign{##\cr
    \hfil$#1\operator@font colim$\hfil\cr
    \noalign{\nointerlineskip\kern1.5\ex@}#2\cr
    \noalign{\nointerlineskip\kern-\ex@}\cr}}%
}
\newcommand{\colim}{%
  \mathop{\mathpalette\colim@{\rightarrowfill@\textstyle}}\nmlimits@
}
\makeatother

\newtheorem{lm}{Lemma}[section]
\newtheorem{lm*}{Lemma}
\newtheorem*{tm*}{Theorem}
\newtheorem*{tms*}{Satz}
\newtheorem{tm}[lm]{Theorem}
\newtheorem{prop}[lm]{Lemma}
\newtheorem*{prop*}{Proposition}
\newtheorem{prob}{Problem}
\newtheorem{cl}[lm]{Corollary}
\newtheorem*{cor*}{Corollary}
\newtheorem{conj}{Conjecture}
\theoremstyle{remark}
\newtheorem*{rk*}{Remark}
\newtheorem*{rm*}{Remark}
\newtheorem{rk}[lm]{Remark}
\newtheorem*{xm}{Example}
\theoremstyle{definition}
\newtheorem{df}{Definition}
\newtheorem*{nt}{Notation}
\newtheorem{Def}[lm]{Definition}
\newtheorem*{Def-intro}{Definition}
\newtheorem{Rk}[lm]{Remark}
\newtheorem{Ex}[lm]{Example}

\theoremstyle{plain}
\newtheorem{Th}[lm]{Theorem}
\newtheorem*{Th*}{Theorem}
\newtheorem*{Th-intro}{Theorem}
\newtheorem{Prop}[lm]{Lemma}
\newtheorem*{Prop*}{Proposition}
\newtheorem{Cr}[lm]{Corollary}
\newtheorem{Lm}[lm]{Lemma}
\newtheorem*{Conj}{Conjecture}
\newtheorem*{BigTh}{Classification of Operations Theorem  (COT)}
\newtheorem*{BigTh-add}{Algebraic Classification of Additive Operations Theorem  (CAOT)}

\newtheorem{maintheorem}{Theorem}
\renewcommand{\themaintheorem}{\Alph{maintheorem}}

\tikzcdset{
arrow style=tikz,
diagrams={>={Straight Barb[scale=0.8]}}
}

\title[Bounded generation of Steinberg groups]{
Bounded generation of Steinberg groups\\ over Dedekind  rings of arithmetic type}

\author{Boris Kunyavski\u\i }
\address{%
Dept. of Mathematics \\
Bar-Ilan University \\
Ramat Gan, Israel
}
\email{kunyav@macs.biu.ac.il}

\author{Andrei Lavrenov }
\address{%
Mathematisches Institut\\
 Universit\"at M\"unchen\\ 
 M\"unchen, Germany
}
\email{avlavrenov@gmail.com}

\author{Eugene Plotkin}
\address{%
Dept. of Mathematics \\
Bar-Ilan University \\
Ramat Gan, Israel
}
 \email{plotkin@macs.biu.ac.il}

\author{Nikolai Vavilov}
\address{%
Dept. of Mathematics and Computer Science \\
St Petersburg State University \\
St Petersburg, Russia
}
\email{nikolai-vavilov@yandex.ru}




\thanks{Research of Boris Kunyavski\u\i , Andrei Lavrenov and Eugene Plotkin was supported by the ISF grant 1994/20.
Nikolai Vavilov was supported by the ``Basis'' Foundation grant
N.\,20-7-1-27-1 ``Higher symbols in algebraic K-theory''. A part of this research was accomplished when Boris Kunyavski\u\i \ was visiting the IHES (Bures-sur-Yvette). Support of these institutions is gratefully acknowledged}

\begin{abstract}
The main 
result of the present paper is bounded elementary generation of the Steinberg groups
$\St(\Phi,R)$ for simply laced root systems $\Phi$
of rank $\ge 2$ and arbitrary Dedekind rings
of arithmetic type. 
Also, we prove bounded generation of
$\St(\Phi,\GF_{q}[t,\,t^{-1}])$ for all root systems
$\Phi$, and bounded generation of
$\St(\Phi,\GF_{q}[t])$ for all root systems $\Phi\neq\mathsf A_1$.

The proofs are based 
on a theorem on bounded elementary generation for the corresponding Chevalley groups, where we provide 
uniform bounds. 

\end{abstract}

\maketitle

\section*{Introduction}

In the present paper, we consider simply-connected Chevalley groups
$G=\mathrm G_{\mathrm{sc}}(\Phi,R)$, 
and
the corresponding Steinberg groups $\St(\Phi,R)$
over Dedekind rings
of arithmetic type.
$G$ is generated by the elementary root unipotents $x_{\alpha}(r)$, $\alpha\in\Phi$, $r\in R$,
and
we are interested in the classical problem of estimating
the width of $G$ 
with respect to the 
generators $x_{\alpha}(r)$. 
The width is defined as the smallest possible $m$ such as every
element of $G$ 
is representable as a product of
$m$ generators $x_{\alpha}(r)$. If there is no such $m$, we
say that the width is infinite.
 If the width is finite, we say that $G$ is {\bf boundedly elementarily generated}.
\par
The problem of bounded generation has attracted considerable attention over
the last 40 years or so. We refer the reader to \cite{KPV}
containing a survey of this long activity as well as some
applications to Kac--Moody groups and model theory.
\par
To make a long story short, given a reduced irreducible root system $\Phi$
of rank $\ge 2$, a Dedekind ring of arithmetic type $R$, and a Chevalley group $G=\mathrm G_{\mathrm{sc}}(\Phi,R)$,
until now there were many cases where it was known that $G$ is boundedly elementarily generated, and two kinds
of general upper estimates for the elementary width of $G$ were available:

$\bullet$ explicit estimates depending on $\Phi$ and the fraction field of $R$;

$\bullet$ estimates depending on $\Phi$ alone in the case $\Phi ={\A}_l$.

Combining the methods of~\cite{KPV} and~\cite{trost2},
we are now able to come up with a complete solution in
the general case:


\begin{maintheorem}[\cite{KPV2}]\label{thm A}
Let\/ $\Phi$ be a reduced irreducible root system of rank $l\geq 2$, 
Then there exists a constant $L=L(\Phi)$, depending on $\Phi$
alone, such that for any Dedekind ring of arithmetic type $R$,
any element in $\mathrm{G}_{\mathrm{sc}}(\Phi,\,R)$ is a product of at most $L$ elementary root unipotents. If the multiplicative group $R^*$ is infinite then the restriction on the rank $l\geq 2$ can be dropped.
\end{maintheorem}

An important --- and unexpected! ---
aspect of this result is the existence of {\it explicit uniform\/} bounds in the function case. In the number
case the bounds are also uniform, but if we wish to cover
all $R$ and not just those with infinite multiplicative
group $R^*$, they are not explicit. Note that
for symplectic groups the  result is new even
in the number case.

We will sketch the proof of Theorem \ref{thm A} in Section \ref{Sec_2}. However, since the goal of the present work 
is bounded generation for Steinberg groups, we intentionally do not provide any computational arguments, especially taking into account 
that in  the most tricky case $\Phi=\mathsf C_2$ they are long and tedious. We refer to \cite{KPV2} for all details and discussions of the proof 
of Theorem \ref{thm A}, taking uniform bounded generation for Chevalley groups for granted. Note, however, that for the sake of 
completeness we collect in Theorem E all cases of Theorem A needed for the treatment of Steinberg groups and provide a full proof.

At this point it is natural to ask whether Theorem~A
or maybe its weaker forms can be generalised to
Steinberg groups. This question was explicitly raised
by Alexei Myasnikov at the conference GAGTA 2022. The reason was that bounded  generation 
of Steinberg groups would have important model-theoretic
applications. It  plays a crucial role in such problems as first order rigidity, elementary equivalence of groups, Diophantine theory, and the like. So, having Theorem~A for Chevalley groups 
as a base, one can think about the similar properties for their coverings.

Again, we are interested in the bounded generation in terms of the set
$$ X=\{x_\alpha(r)\mid \alpha\in\Phi,\ r\in R\} $$
\noindent
of elementary generators (which we continue to
denote by the same letter).
\par\smallskip
However, this case turned out to be much more challenging.
Apart from the bounded generation of the Chevalley groups themselves, it depends on the deep results on the finiteness
of the (linear) $\K_2$-functor, and on bunch of other
difficult results of $\K$-theory, such as stability theorem for
$\K_2$, centrality of $\K_2$, etc.
\par
It is not clear how one could get {\it uniform\/} bounds
in this case. Even with the bounds that depend on $R$
so far we could only prove it for the case when
the  root system $\Phi$ is {\it simply-laced\/}, i.e.,
$\Phi=\mathsf A_l,\mathsf D_l,\mathsf E_l$.


\par

\begin{maintheorem}
Let $\Phi$ be a reduced irreducible simply laced root system
of rank $\ge 2$, and let $R$ be a Dedekind ring of arithmetic type.
If $\Phi=\mathsf A_2$ assume additionally that $R^*$ is infinite. Then $\mathrm{St}(\Phi,\,R)$ is boundedly elementarily generated.
\end{maintheorem}

The idea is to derive this result from Theorem~A. It suffices
to establish that the kernel
$\K_2(\Phi,R)$ of the projection $\St(\Phi,R)\to\mathrm G_{\mathrm{sc}}(\Phi,R)$ is
finite and central, and thus bounded elementary generation of $\mathrm G_{\mathrm{sc}}(\Phi,R)$
implies that of $\St(\Phi,R)$. Here are the main sources on which we rely in this proof.
\par\smallskip
$\bullet$ The {\it stable\/} linear $\K_2(R)$ is finite, for the
function case this is
proven by Hyman Bass and John Tate~\cite{BT}
and for the number case by Howard Garland~\cite{garland}.
(These finiteness results were generalised to higher $\K$-theory
by Daniel Quillen and G\"unter Harder, see the survey by Chuck Weibel~\cite{Kbook}).
\par\smallskip
$\bullet$ However, we need similar results for the unstable
$\K_2$-functors $\K_2(\Phi,R)$. For the {\it linear\/} case
$\SL(n,R)$
there is a definitive stability theorem by Andrei Suslin and
Marat Tulenbaev~\cite{ST}. However, injective stability for
Dedekind rings only
starts with $n\ge 4$, so that for $\SL(3,R)$ one has to refer to
Wilberd van der Kallen~\cite{vdKded} instead, which accounts for the
extra-condition in this case.
\par\smallskip
$\bullet$ 
For other embeddings there are no
stability theorems in the form we need them and starting
where we
want them to start. For instance, in the even orthogonal case
the theorem of Ivan Panin
\cite{panin} starts with $\Spin(10,R)$, whereas we would like to
cover also $\Spin(8,R)$. In any case, there are no similar
results for the exceptional embeddings.
\par\smallskip
Thus, we have to prove a comparison theorem relating $\K_2(\Phi,R)$ to $\K_2(\A_3,R)$. This is obtained as a corollary of partial stability results for Dedekind rings developed by Hideya Matsumoto~\cite{mats} and {\it surjective\/} stability of $\K_2$ for some embeddings, established by Michael Stein~\cite{stein} and the third author~\cite{Pl1,Pl2}.
\par
We also remark that the centrality of $\K_2$ for all Chevalley groups over arbitrary rings is accomplished by
the second author, Sergey Sinchuk and Egor Voronetsky, also in collaboration~\cite{L-another,S-amalgams,LS-another,LS,V-unitary,LSV-centrality}.
\par\smallskip
$\bullet$
An essential obstacle in the symplectic case is that $\K_2(\C_l,R)$
is the Milnor--Witt $\K_2^{\mathrm{MW}}$, rather than the usual Milnor
$\K_2^{\mathrm M}$, as for all other cases (compare \cite{suslin} for an
explicit connection between $\mathrm{K_2Sp}(R)$ and $\K_2(R)$). As is well known, it may fail
to be finite, which means that our approach does not work at all
in this case. This does not mean that the result itself fails, but the
proof would require an entirely different idea.
\par\smallskip
But even for non-symplectic
multiply laced systems, where our approach could theoretically
work, we were unable to overcome occurring technical
difficulties related to the $\K_2$-stability and comparison
theorems. At least, as yet.
\par\smallskip

In contrast to Theorem B, in our second generalisation of Theorem A, we put the restrictions on the ring rather
than the root system.

Namely, using specific calculations of $\K_2(\Phi,\GF_{q}[t])$
and $\K_2(\Phi,\GF_{q}[t,t^{-1}])$ by Eiichi Abe, Jun Morita,
J\"urgen Hurrelbrink and Ulf Rehmann~\cite{abe-morita,
hurrelbrink, morita-rehmann, rehmann} we were able to
establish similar results over $\GF_{q}[t]$ and $\GF_{q}[t,\,t^{-1}]$ also for the multiply laced systems,
even the symplectic ones.
\begin{maintheorem}
Let $\Phi$ be a reduced irreducible root system, and $R=\GF_{q}[t,\,t^{-1}]$ or $R=\GF_{q}[t]$. In the latter case assume additionally that $\Phi\neq\mathsf A_1$. Then $\St(\Phi,\,R)$ is boundedly elementarily generated.
\end{maintheorem}

The paper is organised as follows. In \S~\ref{prelim}
we recall notation and collect some
preliminary results.
In \S~2 we give a sketch of the proof of
Theorem A. 
In \S~3 we recall some basic facts concerning
$K_2(\Phi,R)$ and in \S~4 collect the necessary
facts concerning finiteness of $\mathrm K_2$ in the arithmetic
case. In \S~5 we prove comparison theorems for
$\mathrm K_2(\Phi,R)$ in the case of simply laced $\Phi$,
and thus prove Theorem B. In \S~6 we recall
computation of $\mathrm K_2$ of polynomial rings,
which imply Theorem C. Finally, in \S~7 we mention
some further generalisations.


\section{Notation and preliminaries}  \label{prelim}

In this section we briefly recall the notation that will be used throughout the paper and some background
facts. For more details on Chevalley groups over rings
see \cite{Vav91} or \cite{VP}, where one can
find many further references.

\subsection{Chevalley groups}
\def\wP{\mathcal P}
\def\wQ{\mathcal Q}
Let $\Phi$ be a reduced root system 
and $W=W(\Phi)$ be its Weyl group. In our main
results, $\Phi$ will be assumed irreducible, though
in some proofs one has to use subsystems that are
not. As usual, we choose an order on $\Phi$
and let $\Phi^+$, $\Phi^{-}$ and
$\Pi=\big\{\alpha_1,\ldots,\alpha_l\big\}$ be the corresponding
sets of positive, negative and fundamental roots, respectively.
Further, we consider a lattice $\wP$ intermediate
between the root lattice $\wQ(\Phi)$ and the weight
lattice $\wP(\Phi)$. Finally, let $R$ be a commutative ring with 1,
with the multiplicative group $R^*$.
\par
These data determine the Chevalley group $G=\mathrm G_{\wP}(\Phi,R)$,
of type $(\Phi,\wP)$ over $R$. It is usually constructed as the
group of $R$-points of the Chevalley--Demazure
group scheme $\mathrm G_{\wP}(\Phi,\text{$-$})$ of type $(\Phi,\wP)$.
In the case
$\wP=\wP(\Phi)$ the group $G$ is called simply connected and
is denoted by $\mathrm G_{\sic}(\Phi,R)$. In another extreme case
$\wP=\wQ(\Phi)$ the group $G$ is called adjoint and
is denoted by $\mathrm G_{\ad}(\Phi,R)$.
\par
Many results do not depend
on the lattice $\wP$ and hold for all groups of a given
type $\Phi$.
In all such cases, or when $\wP$ is determined by
the context, we omit any reference to $\wP$ in the notation
and denote by $\mathrm G(\Phi,R)$ {\it any} Chevalley group of type
$\Phi$ over $R$.
However in some cases specific bounds
may depend on $\wP$. Usually, 
we work with a
simply connected group, but in some cases
it is convenient to work with the adjoint group, which
is then reflected in the notation.
\par
In what follows, we also fix a split maximal torus $T=\mathrm T(\Phi,R)$ in $G=\mathrm G(\Phi,R)$ and identify $\Phi$ with $\Phi(G,T)$. This choice
uniquely determines the unipotent root subgroups, $X_{\alpha}$,
$\alpha\in\Phi$, in $G$, elementary with respect to $T$. As usual,
we fix maps $x_{\alpha}\colon R\mapsto X_{\alpha}$, so that
$X_{\alpha}=\{x_{\alpha}(r)\mid r\in R\}$, and require that these parametrisations are interrelated by the Chevalley commutator formula with integer coefficients, see \cite{Carter},
\cite{Steinberg}. The above unipotent elements
$x_{\alpha}(r)$, where $\alpha\in\Phi$, $r\in R$,
elementary with respect to $\mathrm T(\Phi,R)$, are also called
[elementary] unipotent root elements or, for short, simply
root unipotents.
\par
Further,
$$ \mathrm E(\Phi,R)=\big\langle x_\alpha(r)\mid \alpha\in\Phi,\ r\in R\big\rangle $$
\noindent
denotes the {\it absolute\/} elementary subgroup of $\mathrm G(\Phi,R)$,
spanned by all elementary root unipotents, or, what is the
same, by all [elementary] root subgroups $X_{\alpha}$,
$\alpha\in\Phi$.
For $\epsilon\in\{+,-\}$ denote
$$
U^\epsilon(\Phi,\,R)=\left\langle x_\alpha(r)\mid\alpha\in\Phi^\epsilon,\ r\in R\right\rangle\leq\mathrm{E_{}}(\Phi,\,R).
$$

\subsection{Steinberg groups}
Denote by $\mathrm{St}(\Phi,\,-)$ the {\it Steinberg group\/} functor corresponding to $\Phi$. For $\Phi$ that does not have irreducible components
$\cong\mathsf A_1$, and a commutative ring $R$ the Steinberg group $\mathrm{St}(\Phi,\,R)$ is a group defined by the set of generators
$$
\{x_\alpha(r)\mid\alpha\in\Phi,\ r\in R\}
$$
subject to the {\it Steinberg relations\/}
\par\smallskip
$\bullet$ Additivity
$$
x_\alpha(r)x_\alpha(s)=x_{\alpha}(r+s)\quad\text{for}\,\ \alpha\in\Phi,\ r,s\in R, \label{add-relation} $$
\par\smallskip
$\bullet$ Chevalley commutator formula
\begin{equation} 
[x_\alpha(r),\,x_\beta(s)]=\prod_{\substack{i,j\in\mathbb N\setminus0\\i\alpha+j\beta\in\Phi}}x_{i\alpha+j\beta}(N_{\alpha\beta ij\,}r^is^j)\quad\text{for}\,\ \alpha,\beta\in\Phi,\ \beta\neq-\alpha,\ r,s\in R,
\label{chevalley-commutator} 
\end{equation} 
\noindent
where, as usual, $[g,\,h]=ghg^{-1}h^{-1}$ denotes
the left normed commutator, whereas $N_{\alpha\beta ij}\in\mathbb Z$ are the structure constants of the Chevalley group $\mathrm G_{\mathrm{sc}}(\Phi,\,R)$.
\par
The choice of the structure constants $N_{\alpha\beta ij}\in\mathbb Z$ and the order of factors in~(\ref{chevalley-commutator}) are not unique, and we fix any possible choice, see~\cite{VP, Vav08} for many more details and
further references. It is not a problem to specify signs
for classical cases, see \cite{Bour}. On the other hand
in \cite{Vav01} one can find specific choice of the
structure constants $N_{\alpha\beta}$ for
$\mathsf E_6$, $\mathsf E_7$ and $\mathsf E_8$,
corresponding to a positive Chevalley base
(in this case automatically $i=j=1$, so that
$N_{\alpha\beta11}=N_{\alpha\beta}$ are just the
structure constants of the corresponding Lie algebra).
All structure constants $N_{\alpha\beta ij}$ for
$\mathsf F_4$ and $\mathsf G_2$ are tabulated in
\cite{VP}.

\par\smallskip
$\bullet$ For $\mathsf A_1$ one needs another
relation
$$
w_\alpha(u)x_\alpha(r)w_\alpha(u)^{-1}=x_{-\alpha}(-u^{-2}r)\quad\text{for}\,\ \alpha\in\Phi,\ u\in R^*,\ r\in R, \label{a1-relation} $$
\noindent
where
\begin{equation}
\label{def-w}
w_\alpha(u)=x_\alpha(u)x_{-\alpha}(-u^{-1})x_\alpha(u).
\end{equation}

\begin{rk*}
If $\Phi$ does not have irreducible components
$\cong\mathsf A_1$ this extra relation follows from
additivity and the Chevalley commutator formula.
\end{rk*}

\subsection{Arithmetic case}

Let $F$ be a global field and $S$ be a finite non-empty set of places of $F$ containing all archimedean places when $F$ is a number field. Following~\cite{BMS} we will say that
$$
R=\{x\in F\mid v(x)\geq0\ \forall v\not\in S\}
$$
is {\it the Dedekind ring of arithmetic type} defined by the set $S$. Obviously, $R$ is indeed a Dedekind domain, and one can canonically identify the maximal ideals of $R$ with the places outside $S$.

The following result proven by Matsumoto in~\cite[Th\'eor\`eme~12.7]{mats} explains why
we usually prefer to work with simply connected
groups.
\begin{lm}
\label{e=g}
Let $R$ be a Dedekind ring of arithmetic type and
$\Phi$ a reduced irreducible root system of rank at
least $2$. Then
$$
\mathrm E_{\mathrm{sc}}(\Phi,\,R)=\mathrm G_{\mathrm{sc}}(\Phi,\,R).
$$
\end{lm}
In fact, for $\Phi=\mathsf A_l,\,\mathsf C_l$
this result was established already by Hyman Bass, John Milnor
and Jean-Pierre Serre in~\cite{BMS}.
Recently Anastasia Stavrova generalised it to isotropic reductive groups and to polynomial rings over $R$, see~\cite[Corollary~1.2]{stavrova}.


\section{Uniform bounded generation of Chevalley groups:\\ around the proof of Theorem A}\label{Sec_2}


In this section we sketch a proof of Theorem A.  


Recall that the results of~\cite{Mor, MRS, trost2, KMR} completely
solve the problem of the uniform bounded elementary generation
for the special linear groups $\SL(n,R)$, $n\ge 3$, ---
and when $R^*$ is infinite, even for $\SL(2,R)$.
\par
Observe that the methods of ~\cite{KPV} completely reduce the proof
of similar result for almost all other Chevalley groups, including even the {\it symplectic\/} groups $\Sp(2l,R)$, $l\ge 3$, to the case of $\Phi=\A_2$.
\par
{\it The only\/} case
that does not follow rightaway from results
of the above papers, is that of $\Sp(4,R)$.
The analysis of that case is longer and far too technical,
its inclusion would tilt the balance of the present paper. 
Therefore, below we give a complete argument only for the cases 
needed for the proof of Theorem B, collecting these cases in Theorem E. 
The proof for all other cases, with all details and explicit bounds, 
is contained in \cite{KPV2}.


\subsection{Tavgen rank reduction theorem}
In most cases the reduction to $\A_1$ or $\A_2$
is based on the following cunning observation, whose idea
goes back to the work of Oleg Tavgen \cite{Tavgen}.
His trick was then generalised in \cite{VSS} and
\cite{SSV}. The following final form is proven in 
{\cite[Theorem~3.2]{KPV}}.

\begin{lm}\label{tavgen}
Let $\Phi$ be a reduced irreducible root system of rank $l\geq2$, and $R$ be a commutative ring. Let $\Delta_1,\ldots,\Delta_t$ be some subsystems of $\Phi$, whose union contains all fundamental roots of $\Phi$. Suppose that for all $\Delta_i$ the elementary Chevalley group $\mathrm{E_{}}(\Delta_i,\,R)$ admits a unitriangular factorisation
$$
\mathrm{E_{}}(\Delta_i,\,R)=U^+(\Delta_i,\,R)\,U^-(\Delta_i,\,R)\,U^+(\Delta_i,\,R)\ldots U^{\pm}(\Delta_i,\,R)
$$
of length $N$ {\rm(}not depending on $i${\rm)}. Then the elementary group $\mathrm{E_{}}(\Phi,\,R)$ itself admits unitriangular factorisation
$$
\mathrm{E_{}}(\Phi,\,R)=U^+(\Phi,\,R)\,U^-(\Phi,\,R)\,U^+(\Phi,\,R)\ldots U^{\pm}(\Phi,\,R)
$$
of the same length $N$.
\end{lm}

Below, we essentially apply it to two cases, when
all $\Delta_i$'s are $\A_1$, and when all of them are
$\A_2$.


\subsection{The case when $R^*$ is infinite}
The case where $R$ has infinitely many units
and its field of fractions is a number field
is {\it completely\/} solved, with very small {\it absolute} constant. We cannot describe the whole chain of events here, and mention all contributors.
After the initial breakthrough by Maxim
Vsemirnov \cite{Vs}, which was a first unconditional
result of this sort, not depending on the GRH,
Aleksander Morgan, Andrei Rapinchuk and Sury
\cite{MRS} succeeded in solving
the number case, with the bound $L=9$.
This bound was then improved to $L=8$ by
Bruce Jordan and Yevgeny Zaytman \cite{JZ}
(and can be further improved in the presence of
finite or real valuations in $S$).

\begin{lm}\cite{JZ}
\label{sury}
For any Dedekind ring of arithmetic type $R$ in a number field  with
the infinite multiplicative group $R^*$ any element in
$\SL(2,R)$ is a product of at most $8$ elementary transvections. 
\end{lm}

\begin{rk} \label{rem-rk1} 
In the paper presently under way
the first author, Dave Morris and Andrei Rapinchuk
\cite{KMR} improve the bound to $L=7$ in the number case (which we believe is the best possible and cannot
be further improved, in general). Also, they obtain a
similar result in the function case, with the bound
$L=8$ (which, we believe, can be further improved to $L=7$).
\end{rk}


\begin{maintheorem}[\cite{KMR}]
\label{infinite}
For any Dedekind ring of arithmetic type $R$ 
with the infinite multiplicative group $R^*$ any element in
$\mathrm G_{\mathrm{sc}}(\Phi,\,R)$ is a product of at most $L=8|\Phi^+|$ elementary unipotents.
\end{maintheorem}

\begin{proof}
Combine Lemma \ref{tavgen} with $\Delta_i=A_1$, Lemma \ref{sury} and Remark \ref{rem-rk1}. 
\end{proof}


Thus, if we are not interested in actual bounds, but
just in uniform boundedness, 
one can restrict oneself to considering 
the Dedekind rings of arithmetic type 
with {\it finite\/} multiplicative groups. In the number case, these are $\mathbb Z$ and the rings of
integers in imaginary quadratic number fields. 
Note that as discovered by Alexander Trost \cite{trost2}, in
the function case for ranks $\ge 2$ we
do not have to distinguish between rings with finite
and infinite multiplicative group. 

\subsection{The simply laced case and $\Phi=\mathsf F_4$}

\begin{lm}~\cite[Theorem~4.1]{trost2}
\label{trost}
For each $l\ge 2$, there exists a constant $L=L(l)\in\mathbb N$ such that for any Dedekind ring of arithmetic type $R$, any element in $\mathrm{G}_{\mathrm{sc}}(\mathsf A_{l},\,R)$ is a product of at most $L$ elementary root unipotents.
\end{lm}

\begin{proof}

By a theorem of Carter--Keller--Paige  (redeveloped
by Morris \cite{Mor}), see \cite[(2.4)]{CKP}, 
bounded generation for groups of type $\mathsf A_{l}$, $l\ge 2$, holds for Dedekind rings $R$ in number fields $K$, with a bound depending on $l$ and also on the degree $d$ of $K$. But since for all degrees $d\ge 3$
the existence of a uniform bound already follows from
Theorem D, we only need to take maximum of that, and
the universal bound for $d=1,2$.
\par
Combining this result with the subsequent work of Trost~\cite{trost2} on the function field case,
one obtains
the  result.
\end{proof}

\begin{rk}
In fact, in the sequel we only need the special case
of the above result pertaining to $\SL(3,R)$, which
corresponds to $\Phi=\mathsf A_2$. 
In the function case Trost \cite{trost2} gave the estimate $L(2)\le 65$.
No such explicit estimate is known in the number case.

\end{rk}

\par
Since the fundamental systems of the simply laced
systems and $\F_4$ are covered by copies of $\A_2$,
combining Lemma~\ref{tavgen} with Lemma \ref{trost} one gets
another stronger form of Theorem A, now without the
assumption that $R^*$ is infinite, but only in the
special case of simply laced systems of rank $\ge 2$ and
$\mathsf F_4$. This is the only part of Theorem A
on which Theorem B relies.

\begin{maintheorem}
Let $\Phi$ be simply laced of rank $\ge 2$ or
$\Phi=\mathsf F_4$ and $R$
be any Dedekind ring of arithmetic type. Then
$\mathrm G_{\mathrm{sc}}(\Phi,\,R)$ is a product of at most $L(2)\cdot|\Phi^+|$ elementary unipotents.
\end{maintheorem}

The bound here is very rough, since $L(2)$ is the
number of {\it elementary\/} factors, the number of
unitriangular ones can be much smaller. Also, the
use of stability allows to get much better bounds,
of the type $L=L(2)+4|\Phi^+|$, 
where some multiple of $|\Phi^+|$ occurs as a summand,
not as a factor. 


\subsection{Idea of the rest of the proof of Theorem A}
\label{sec:stab}

To establish Theorem~A for the cases of Chevalley groups $\mathrm G(\Phi,R)$ not covered by Theorem~E, where Tavgen's trick played a crucial role, 
we use the arguments based on the surjective stability of K$_1$-functor, in the spirit of \cite{KPV}.  
This way, we obtain the following reduction theorem.

\begin{lm}\label{reduction}
Let $R$ be a  Dedekind ring of arithmetic type. Then (uniform) bounded generation of the groups $\mathrm G(\Phi,R)$, $\Phi\neq \mathsf C_2$ 
follows from (uniform) bounded generation of the group $\mathrm G(\mathsf A_2,R)$. \qed 
\end{lm} 

For instance, the case $\mathrm{G}(\mathsf G_2,\,R)$ follows from \cite[Proposition~4.3]{KPV}. 
By stability arguments, $\mathrm G(\mathsf B_l,R)$ and $\mathrm G(\mathsf C_l,R)$ $(l\ge 3)$ reduce to $\mathrm G(\mathsf B_3,R)$ and $\mathrm G(\mathsf C_3,R)$, 
respectively. The group $\mathrm G(\mathsf B_3,R)$ is treated as in \cite[Section~6.2]{KPV}. Somewhat surprisingly, 
the case $\mathrm G(\mathsf C_3,R)$ can be reduced to $\mathrm G(\mathsf A_2,R)$, along the lines of \cite[Sections 5 and 6]{KPV}, 
using arithmetic lemmas of Carter and Keller \cite{CaKe} in the number case and of Trost \cite{trost2} 
in the function case. Actually, the idea of such a reduction
was contained already in Zakiryanov's thesis, see \cite{Zak}, but the authors of \cite{KPV}  have not realised this fact before
rediscovering the same idea in the general case in
March 2023\footnote{It is wrongly claimed in
\cite{Zak} that $\Sp(4,\mathbb Z)$ is not boundedly
generated. As a result this work has not been given
the credit it deserves. In particular, it should have been cited in the historical survey of \cite{KPV}.}.

Note that in all cases we get absolute constants as bounds, depending only on $L(2)$ appearing in Lemma 
\ref{trost}. Thus these bounds are explicit in the function case and implicit in the case where $R$ is 
a quadratic imaginary ring.

This finishes the proof of Theorem A in all cases except $\mathsf C_2$ which turns out to be 
much more involved. It does not reduce to $\mathrm G(\mathsf A_2,R)$ and is settled in \cite{KPV2} 
by the same methods as in \cite{Mor, trost, trost2}, using some results of \cite{KPV}. As in all other 
cases, the obtained bounds are in the form of absolute constants, which are explicit 
in the function case.


\section{$\mathrm K_2$ modeled on Chevalley groups}

In this section we collect the classical results on $\mathrm K_2(\Phi,\,-)$ which we will use in this paper.

There is a natural map from $\mathrm{St}(\Phi,\,R)$ to $\mathrm{G_{sc}}(\Phi,\,R)$ sending generators of the Steinberg group $x_\alpha(r)$ to elementary root unipotents $x_\alpha(r)$ of the Chevalley group, $\alpha\in\Phi$, $r\in R$. Following~\cite{stein}, we denote
$$
\mathrm K_2(\Phi,\,R)=\mathrm{Ker}\big(\mathrm{St}(\Phi,\,R)\rightarrow\mathrm G_{\mathrm{sc}}(\Phi,\,R)\big).
$$
\par
Let $R$ be a commutative ring. Following Steinberg, for
$\alpha\in\Phi$, $u\in R^*$ we define the elements
\begin{align}
\label{def-sym}
h_\alpha(u)=w_\alpha(u)w_\alpha(-1)
\end{align}
of the Steinberg group $\mathrm{St}(\Phi,R)$, where $w_\alpha(u)$ is defined in \S~1\,(\ref{def-w}),
see~\cite{Sgrc}.
\par
Further, for two invertible elements
$u,\,v\in R^*$ we define the {\bf Steinberg symbol}
$$ \{u,\,v\}_\alpha=h_{\alpha}(uv)h_\alpha(u)^{-1}h_\alpha(v)^{-1}. $$

The following fact is well-known, see, for instance,~\cite[Proposition~1.3]{Sdim0}.
\begin{prop}
\label{sym-cent}
For a ring $R$ and a reduced irreducible root system $\Phi$ elements $\{u,\,v\}_\alpha$ for $u,\,v\in R^*$, $\alpha\in\Phi$, are central in $\mathrm{St}(\Phi,\,R)$ and belong to $\mathrm K_2(\Phi,\,R)$.
\end{prop}

The following classical result is due to Matsumoto~\cite[Corollaire~5.11]{mats}.
\begin{prop}
\label{mats} Let $k$ be a field, $\Phi$ be a reduced irreducible root system.\\
{\rm 1)} The group $\mathrm K_2(\Phi,\,k)$ is generated by $\{u,\,v\}_\alpha$ for any fixed {\rm long} root $\alpha\in\Phi$, and all $u,\,v\in R^*$.\\
{\rm 2)} Let $\Phi$ be a non-symplectic reduced irreducible root system {\rm(}i.e., $\Phi\neq\mathsf A_1,\,\mathsf B_2,\,\mathsf C_l${\rm)}. Consider any embedding $\mathsf A_2\hookrightarrow\Phi$ on {\rm long} roots. Then the induced map
$$
\mathrm K_2(\mathsf A_2,\,k)\rightarrow\mathrm K_2(\Phi,\,k)
$$
is in fact an isomorphism.\\
{\rm 3)} For a symplectic reduced irreducible root system $\mathsf C_1=\mathsf A_1$, $\mathsf C_2=\mathsf B_2$, or $\mathsf C_l$, $l\geq3$, consider any embedding $\mathsf A_1\hookrightarrow\mathsf C_l$ on {\rm long} roots. Then the induced map
$$
\mathrm K_2(\mathsf A_1,\,k)\rightarrow\mathrm K_2(\mathsf C_l,\,k)
$$
is in fact an isomorphism.
\end{prop}
\begin{rk}
In fact, Matsumoto describes $\mathrm K_2(\Phi,\,k)$ in terms of generators and relations in~\cite[Corollaire~5.11]{mats}. In modern terms, Matsumoto proved that $\mathrm K_2(\Phi,\,k)$ coincides with Milnor $\mathrm K_2^{\mathrm M}(k)$ for non-symplectic $\Phi$ and with Milnor--Witt $\mathrm K_2^{\mathrm{MW}}(k)$ for symplectic $\Phi$. However, we will not need the explicit description of relations in this paper.
\end{rk}

We will also need the following stabilisation results. The next statement is a particular case of the Suslin--Tulenbaev theorem, see~\cite[Corollary~4.2]{ST}.

\begin{prop}
\label{st}
Let $R$ be a Dedekind domain, then the natural map
$$
\mathrm K_2(\mathsf A_l,\,R)\rightarrow\mathrm K_2(\mathsf A_{l+1},\,R)
$$
is surjective for $l\geq2$ and injective for $l\geq3$.
\end{prop}

For the root systems other than $\mathsf A_l$ we
only have the surjective stability part, established by 
Stein~\cite[Corollary~3.2, Theorem~4.1]{stein}.
\begin{prop} 
\label{stein}
Let $R$ be a Dedekind domain. Consider the following embeddings of root systems $\Psi\hookrightarrow\Phi$:
\begin{itemize}
\item
natural embedding $\mathsf D_l\hookrightarrow\mathsf D_{l+1}$ for $l\geq3$;
\item
natural embedding $\mathsf E_l\hookrightarrow\mathsf E_{l+1}$ for $l=6,7$;
\item
the embedding $\mathsf D_5=\{\alpha_1,\alpha_2,\alpha_3,\alpha_4,\alpha_5\}\hookrightarrow\mathsf E_{6}$ {\rm(}with numbering according to Bourbaki~\cite[Table~V]{Bour}{\rm)}.
\end{itemize}
Then the induced map
$$
\mathrm K_2(\Psi,\,R)\rightarrow\mathrm K_2(\Phi,\,R)
$$
is surjective.
\end{prop}

On the other hand, for Dedekind rings of arithmetic
type with infinite multiplicative groups the bounds
in surjective/injective stability can be improved by 1.
This was done by van der Kallen~\cite[Theorem~1]{vdKded}.
\begin{prop} 
\label{vdk}
Let $R$ be a Dedekind ring of arithmetic type with infinitely many units. Then the natural map
$$
\mathrm K_2(\mathsf A_l,\,R)\rightarrow\mathrm K_2(\mathsf A_{l+1},\,R)
$$
is surjective for $l\geq1$ and injective for $l\geq2$.
\end{prop}

Finally, we will need also another result by 
Stein claiming that surjective stability implies
centrality of $\mathrm K_2$~\cite[Theorem~5.1]{Sgrc}.

\begin{prop} 
\label{central}
Let $\Pi$ denote a set of simple roots in a reduced root system $\Phi$. For $\alpha\in\Pi$, let $\Psi\subseteq\Phi$ be the subsystem generated by $\Pi\setminus\alpha$. Then
$$
\mathrm K_2(\Phi,\,R)\cap\mathrm{Im}\big(\mathrm{St}(\Psi,\,R)\rightarrow\mathrm{St}(\Phi,\,R)\big)
$$
is a central subgroup of $\mathrm{St}(\Phi,\,R)$ for any commutative ring $R$.
\end{prop}


\section{Stable linear $\mathrm K_2$}
\label{stable-case}

Recall that
$$ \mathrm K_2(R)=\lim\limits_{l\rightarrow\infty}\mathrm K_2(\mathsf A_l,\,R) $$
is the usual stable linear 
$\mathrm K_2$-functor for any ring $R$ (cf.~\cite[Chapter~III, Section~5]{Kbook}).

For a place $v$ of a field $F$ let $\kappa_v$ denote the corresponding residue class field and
$$
\partial_v\colon\mathrm K_2(F)\rightarrow\kappa_v^*
$$
the corresponding residue homomorphism (also called {\it tame symbol}) sending generators $\{x,\,y\}_\alpha$ of $\mathrm K_2(F)$ (see Proposition~\ref{mats}) to
$$
(-1)^{v(x)v(y)}\overline{\left(\frac{y^{\,v(x)}}{x^{\,v(y)}}\right)}\in\kappa_v^*,
$$
see~\cite[Chapter~III, Lemma~6.3]{Kbook}. We will need the following result due to Christophe Soul\'e, see, for instance,~\cite[Chapter~V, Theorem~6.8]{Kbook}.

\begin{prop} 
\label{quillen}
Let $R$ be a Dedekind domain whose field of fractions $F$ is a global field. Then there is an exact sequence
$$
0\rightarrow\mathrm K_2(R)\rightarrow\mathrm K_2(F)\xrightarrow{\oplus\partial_{\mathfrak p}}\bigoplus_{\mathfrak p}(R/\mathfrak p)^*\rightarrow0,
$$
where the first arrow is induced by the natural inclusion $R\hookrightarrow F$, and the sum is taken over all non-zero prime ideals $\mathfrak p$ of $R$.
\end{prop}

Our proof of Theorem B heavily relies on the following
classical result. In the function case
this is due to Bass and 
Tate~\cite[Chapter~II, Theorem~2.1]{BT}.
In the number case this was first established by Garland~\cite{garland} by analytic methods (see also~\cite[Chapter~II, Remark after Theorem~2.1]{BT}).

\begin{prop} 
\label{tate-garland}
Let $F$ be a global field. Then
$$
\mathrm H_2=\mathrm{Ker}\left(\mathrm K_2(F)\xrightarrow{\oplus\partial_v}\bigoplus\kappa_v^*\right),
$$
where the sum is taken over all {\rm finite} places of $F$, is finite.
\end{prop}

\begin{rk}
The case $\mathrm{char}\,F=p>0$ of Proposition~\ref{tate-garland} was generalised to higher K-theory by Harder~\cite[Korollar~3.2.3]{harder}.
\end{rk}

The following corollary is perhaps also well-known, however we do not know the reference and provide an argument for the convenience of the reader.

\begin{cl}
\label{at-fin}
Let $R$ be a Dedekind ring of arithmetic type defined by the set of places $S$. Then $\mathrm K_2(R)$ is finite.
\end{cl}
\begin{proof}
By Proposition~\ref{quillen} we get an exact sequence
$$
0\rightarrow\mathrm K_2(R)\rightarrow\mathrm K_2(F)\xrightarrow{\oplus\partial_{v}}\bigoplus_{v\not\in S}\kappa_v^*\rightarrow0.
$$
Therefore we may consider $\mathrm  K_2(R)$ as a subgroup of $\mathrm K_2(F)$, and restricting $\partial_v$ to it we get an exact sequence
$$
0\rightarrow\mathrm H_2\rightarrow\mathrm K_2(R)\xrightarrow{\oplus\partial_{v}}\bigoplus_{\substack{v\in S\\\text{finite}}}\kappa_v^*\rightarrow0,
$$
where $\mathrm H_2$ is the group from Proposition~\ref{tate-garland}. However, $S$ is a finite set, $\kappa_v$ is a finite field for a finite place $v$, and $\mathrm H_2$ is finite by Proposition~\ref{tate-garland}.
\end{proof}

We will denote by $I(k)$ the fundamental ideal of the Witt ring of symmetric bilinear forms $W(k)$ of a field $k$, see~\cite{milnor-husemoller}. Following~\cite{suslin, morita-rehmann} we denote $\mathrm K_2\mathrm{Sp}(R)=\lim\limits_l\mathrm K_2(\mathsf C_l,\,R)$. We will need the following result due to Suslin~\cite[Theorem~6.5]{suslin}.

\begin{prop} 
\label{kmw}
For any field $k$ there is an exact sequence
$$
0\rightarrow I^3(k)\rightarrow\mathrm{K_2Sp}(k)\rightarrow\mathrm K_2(k)\rightarrow 0.
$$
\end{prop}


\section{
Comparison theorems for $\mathrm K_2$: proof of theorem B}  \label{secBC}

\subsection{Simply laced root systems}

The aim of this section is to prove the following result, which, together with Theorem A (or, in fact, already with its special case, Theorem E), implies Theorem B.
We believe it is very interesting in its own right, and
may have further applications.

\begin{maintheorem}
\label{result}
Let $R$ be a Dedekind ring of arithmetic type, and let $\Phi\neq\mathsf A_1$ be a simply laced reduced irreducible root system {\rm(}i.e., $\Phi=\mathsf A_l,\,\mathsf D_l,\,\mathsf E_l$, $l\neq1${\rm)}. Assume additionally that
$$
\text{either }\quad \mathrm{rk}(\Phi)\geq3\quad \text{ or }\quad R\,\text{ has infinitely many units}.
$$
Then $\mathrm K_2(\Phi,\,R)$ is a central subgroup of $\mathrm{St}(\Phi,\,R)$, and, moreover,
$$
\mathrm K_2(\Phi,\,R)=\mathrm K_2(R),
$$
in particular, $\mathrm K_2(\Phi,\,R)$ is finite.
\end{maintheorem}

The above theorem was probably never published,
but it may be mostly known to experts. In any case,
for $\Phi=\mathsf D_l$, $l\geq5$, it follows from
a result by Panin \cite[Theorem~6.1]{panin}. Moreover,
in the same paper Panin proves in~\cite[Theorem~9.1]{panin} a similar stabilisation result also for higher orthogonal K-theory.
\par
However, to cover also $\Phi=\mathsf D_4,\mathsf E_6,
\mathsf E_7$
and $\mathsf E_8$ we start with the following
result.

\begin{lm}
\label{theorem}
Let $R$ be a Dedekind ring of arithmetic type, and let $\Phi$ denote a simply laced reduced irreducible root system of $\mathrm{rk}(\Phi)\geq3$. Then there exists an embedding $\mathsf A_3\hookrightarrow\Phi$ such that the induced map
$$
\mathrm K_2(\mathsf A_3,\,R)\rightarrow\mathrm K_2(\Phi,\,R)
$$
is an isomorphism.
\end{lm}
\begin{proof}
By Proposition~\ref{stein} we conclude that there exists an embedding $\mathsf A_3\hookrightarrow\Phi$ such that the induced map
$$
\mathrm K_2(\mathsf A_3,\,R)\rightarrow\mathrm K_2(\Phi,\,R)
$$
is surjective. Let $F$ be the field of fractions of $R$. By Proposition~\ref{mats} we conclude that the induced map
$$
\mathrm K_2(\mathsf A_3,\,F)\rightarrow\mathrm K_2(\Phi,\,F)
$$
is an isomorphism. Moreover, by Proposition~\ref{st} we have $\mathrm K_2(\mathsf A_3,\,R)=\mathrm K_2(R)$, and therefore the natural map
$$
\mathrm K_2(\mathsf A_3,\,R)\rightarrow\mathrm K_2(\mathsf A_3,\,F)
$$
is injective by Proposition~\ref{quillen}. Consider the following commutative diagram:
$$
\begin{tikzcd}
\mathrm K_2(\mathsf A_3,\,R)\ar[two heads]{d}\ar[tail]{r}&\mathrm K_2(\mathsf A_3,\,F)\ar{d}{\cong}\\
\mathrm K_2(\Phi,\,R)\ar{r}&\mathrm K_2(\Phi,\,F).
\end{tikzcd}
$$
The claim follows by a simple diagram chase.
\end{proof}

Now we are all set to finish the proof of Theorem~\ref{result}.

\begin{proof}[Proof of Theorem~\ref{result}]
For $\mathrm{rk}(\Phi)\geq3$ consider the embedding $\mathsf A_3\hookrightarrow\Phi$ from Lemma~\ref{theorem} and use that $\mathrm K_2(\mathsf A_3,\,R)=\mathrm K_2(R)$ by Proposition~\ref{st} to get the equality
$$
\mathrm K_2(\Phi,\,R)=\mathrm K_2(R).
$$
To prove centrality, use the surjectivity of the map $\mathrm K_2(\mathsf A_3,\,R)\rightarrow\mathrm K_2(\Phi,\,R)$ for $\mathrm{rk}(\Phi)\geq4$ or the surjectivity of the map $\mathrm K_2(\mathsf A_2,\,R)\rightarrow\mathrm K_2(\mathsf A_3,\,R)$ from Proposition~\ref{st} together with Proposition~\ref{central}. The finiteness of $\mathrm K_2(R)$ follows from Corollary~\ref{at-fin}.

For $\Phi=\mathsf A_2$ the claim follows from Proposition~\ref{vdk} (together with Proposition~\ref{central} and Corollary~\ref{at-fin}).
\end{proof}

Theorem~B is a direct consequence of Theorem~A
and Theorem~\ref{result}.


\subsection{Multiply laced root systems}
As we see in the next section, for multiply laced
root systems $\Phi$ the equality
$$ \mathrm K_2(\Phi,\,R)=\mathrm K_2(R) $$
\noindent
may hold for some Dedekind rings $R$.
\par
However, the following counter-example shows that
it certainly fails in the symplectic case, in general.

\begin{prop}
\label{cl}
For $l\geq3$ one has
$$
\mathrm K_2(\mathsf C_l,\,\mathbb Z)\neq\mathrm K_2(\mathbb Z).
$$
\end{prop}
\begin{proof}
Recall that $\mathrm K_2(\mathbb Z)=\mathbb Z/2$ (see, e.g.,~\cite[Corollary~10.2]{milnor}). However, the map $\mathrm K_2(\mathsf C_l,\,\mathbb Z)\rightarrow\mathrm K_2(\mathsf C_{l+1},\,\mathbb Z)$ is surjective by~\cite[Corollary~3.2]{stein}, and
$$
\lim_l\mathrm K_2(\mathsf C_l,\,\mathbb Z)=\mathrm{K_2Sp}(\mathbb Z)=\mathbb Z,
$$
see, e.g.,~\cite[Theorem~2.1]{sch}.
\end{proof}

Thus, there is no hope to prove the symplectic
analogue of Theorem B along the same lines. This
does not mean that bounded generation of
$\St(\mathsf C_l,R)$, $l\ge 3$, fails. But if it holds,
its proof would require some completely different ideas.


\section{$K_2$ for polynomial rings: proof of Theorem C}

In this section we consider the polynomial rings $R=\mathbb F_q[X]$ and $R=\mathbb F_q[X,\,X^{-1}]$.
Bounded generation of the Chevalley groups themselves
in these cases is proven in \cite{KPV}. On the other hand,
for these rings $\mathrm K_2(\Phi,\,R)$ is generated by
the usual Steinberg symbols $\{u,\,v\}_\alpha$, $u,\,v\in R^*$,
which allows one to explicitly calculate it.
Observe that this is rarely the case for more general Dedekind domains, where one needs higher symbols.
\par
Recall that $I(k)$ denotes the fundamental ideal of the Witt ring of symmetric bilinear forms of a field $k$ (see~\cite{milnor-husemoller}). We will use the following well-known facts.
\par
The first statement below is the [second] Steinberg theorem,
it is proven, e.g., in \cite{Steinberg} or in~\cite[Corollary~9.13]{milnor}. For the second
statement see, e.g.,~\cite[Chapter~IV, Lemma~1.5]{milnor-husemoller}.

\begin{prop}
\label{finite-fields}
Let $\mathbb F_q$ be a finite field. Then
\begin{enumerate}
\item[{\rm 1)}]
$\mathrm K_2(\mathbb F_q)=0$;
\item[{\rm 2)}]
$I^2(\mathbb F_q)=0$.
\end{enumerate}
\end{prop}

The next result is an immediate corollary of
a result by Rehmann \cite{rehmann}.

\begin{prop} 
\label{poly}
Let $R=\mathbb F_q[X]$ be the polynomial ring over a finite field, $\Phi$ {\it any} reduced irreducible root system. Then
$$
\mathrm K_2(\Phi,\,R)=\mathrm K_2(R)=0.
$$
\end{prop}
\begin{proof}
For any field $k$ the natural embedding induces an isomorphism
$$
\mathrm K_2(\Phi,\,k)\cong\mathrm K_2(\Phi,\,k[t])
$$
by~\cite[Korollar zu Satz~1]{rehmann}. It remains to use that $\mathrm K_2(\mathbb F_q)=0=I^3(\mathbb F_q)$ by Proposition~\ref{finite-fields}, and apply Proposition~\ref{kmw}.
\end{proof}

The key role in the proof of Theorem C is played
by the following observation of the second author and Sinchuk, see~\cite[Lemma~2.2]{LS}, which
in turn relies on deep results of Hurrelbrink,
Abe and Morita.

\begin{lm}  
\label{fundamental}
For an arbitrary field $k$ and a non-symplectic root system $\Phi$ there is an exact sequence of abelian groups
$$
0\rightarrow\mathrm K_2(\Phi,\,k)\rightarrow\mathrm K_2(\Phi,\,k[X,\,X^{-1}])\rightarrow k^*\rightarrow 0
$$
split by the map
$$ \{X,\,-\}_\alpha\colon
k^*\rightarrow\mathrm K_2(\Phi,\,k[X,\,X^{-1}]) $$
\noindent
for any fixed long root $\alpha\in\Phi$. In particular, the natural embedding induces an injective map
$$
\mathrm K_2(\Phi,\,k[X,\,X^{-1}])\hookrightarrow\mathrm K_2(\Phi,\,k(X)).
$$
\end{lm}
\begin{proof}
Since $\mathrm K_2(\Phi,\,F)=\mathrm K_2(F)$ for any field $F$ by Proposition~\ref{mats}, the second statement follows from the first one.

Indeed, the map $\mathrm K_2(k)\rightarrow\mathrm K_2(k(X))$ is injective, e.g., by Milnor's theorem~\cite[Chapter~III, Example~6.1.2, Theorem~7.4]{Kbook}, and $k^*\rightarrow\mathrm K_2(k(X))$ is injective as a splitting to the residue homomorphism $\partial_X$ corresponding to an order of the zero or the pole at $X=0$ (see Section~\ref{stable-case}, cf. also~\cite[Proof of Lemma~2.2]{LS}).

The first statement is proven for $\Phi\neq\mathsf G_2$ in~\cite[Satz~3]{hurrelbrink} (cf.~\cite[Lemma~2.2]{LS}, Proposition~\ref{mats}). For $\Phi=\mathsf G_2$ consider the following commutative diagram
$$
\begin{tikzcd}
\mathrm K_2(\mathsf A_2,\,k[X,\,X^{-1}])\ar[tail]{d}\ar[two heads]{r}&\mathrm K_2(\mathsf G_2,\,k[X,\,X^{-1}])\ar{d}\\
\mathrm K_2(\mathsf A_2,\,k(X))\ar{r}{\cong}&\mathrm K_2(\mathsf G_2,\,k(X)),
\end{tikzcd}
$$
where the horizontal arrows are induced by the natural embedding $\mathsf A_2\hookrightarrow\mathsf G_2$ as a set of long roots.

Since $\mathrm K_2(\mathsf G_2,\,k[X,\,X^{-1}])$ is generated by $\{u,\,v\}_\alpha$ for $u,\,v\in k[X,\,X^{-1}]^*$, and $\alpha\in\mathsf G_2$ a fixed long root by~\cite[Corollary~6]{abe-morita}, we conclude that the top horizontal arrow is surjective. The bottom horizontal arrow is an isomorphism by Proposition~\ref{mats}. The injectivity of left horizontal arrow is already discussed above. Therefore, a simple diagram chase shows that the top horizontal arrow is in fact an isomorphism.
\end{proof}

Modulo this results we can now summarise the
Hurrelbrink, Abe--Morita, and Morita--Rehmann
as follows.

\begin{prop}
\label{laur}
Let $R=\mathbb F_q[X,\,X^{-1}]$ be the Laurent polynomial ring over a finite field, $\Phi$ {\it any} reduced irreducible root system. Then $\mathrm K_2(\Phi,\,R)$ is a central subgroup of $\mathrm{St}(\Phi,\,R)$, and, moreover,
$$
\mathrm K_2(\Phi,\,R)=\mathrm K_2(R)=\mathbb F_q^*.
$$
\end{prop}
\begin{proof}
For any field $k$ there is an isomorphism
$$
\mathrm K_2(\Phi,\,k[X,\,X^{-1}])\cong\mathrm K_2(k)\oplus k^{*}
$$
for $\Phi\neq\mathsf C_l$ 
by Lemma~\ref{fundamental}, and
$$
\mathrm K_2(\mathsf C_l,\,k[X,\,X^{-1}])\cong\mathrm{K_2Sp}(k)\oplus\mathrm P(k)
$$
where $\mathrm P(k)=k^*\oplus I^2(k)$ for $l\geq 1$ by~\cite[Theorem~B]{morita-rehmann}. It remains to observe that $I^2(\mathbb F_q)=0=\mathrm K_2(\mathbb F_q)$ by Proposition~\ref{finite-fields}, and therefore (using Proposition~\ref{kmw}) one has
$$
\mathrm K_2(\Phi,\,R)=\mathrm K_2(R)=\mathbb F_q^*.
$$
To prove the first statement, observe that $\mathrm K_2(\Phi,\,k[X,\,X^{-1}])$ is generated by the Steinberg
symbols $\{u,\,v\}_\alpha$ for $u,\,v\in k[X,\,X^{-1}]^*$ by~\cite[Corollary~6]{abe-morita}, in particular, it is a central subgroup of $\mathrm{St}(\Phi,\,k[X,\,X^{-1}])$ by Proposition~\ref{sym-cent}.
\end{proof}

Now Theorem~C is a direct consequence of
\cite[Theorem A and Theorem C]{KPV}
and Propositions~\ref{poly},~\ref{laur}.


\section{Concluding remarks} 
\label{sec-conj}

Here we mention some eventual generalisations of the
results of the present paper.

\par\smallskip
$\bullet$ Let $I\unlhd R$ be an ideal of $R$.
In the present paper we addressed the {\it absolute\/} case $I=R$ alone. However, it makes sense to ask similar questions for
the {\it relative\/} case, in other words for the relative elementary subgroups $E(\Phi,R,I)$ of level $I\unlhd R$.
(Unlike the absolute case, $E(\Phi,R,I)$ does not
necessarily coincide with the congruence subgroups $G(\Phi,R,I)$ of the same level.)
\par
The expectation is that for  $E(\Phi,R,I)$ one can get similar {\it uniform\/} bounds
in terms of the elementary conjugates
$$ x_{-\alpha}(r)x_{\alpha}(s)x_{-\alpha}(-r),
\qquad \alpha\in\Phi,\quad s\in I,\quad r\in R. $$
\par
Otherwise, one could look at the {\it true\/} =
{\it unrelativised\/} elementary subgroup $E(\Phi,I)$
of level $I$ generated by $x_{\alpha}(s)$,
$\alpha\in\Phi$, $s\in I$, and ask a similar question
in terms of the elementary generators of level $I$.

\begin{prob}
Establish analogues of Theorem A for the
elementary groups $E(\Phi,I)$ and $E(\Phi,R,I)$
of level $I$, with uniform bounds not depending
on either $R$ or $I$.
\end{prob}

Some partial results in this
direction for classical groups are obtained by
Sergei Sinchuk and Andrei Smolensky \cite{SiSm} and
by Pavel Gvozdevsky \cite{Gvoz23}, but their bounds
are not uniform. As a more remote goal one could think
of generalisations to birelative subgroups, see \cite{HSVZ}
for this context.



\par\smallskip
$\bullet$ It is known that bounded elementary generation is closely related to many other flavours of bounded generation, including, in particular, finite width in commutators.
\par
Namely, Alexei Stepanov \cite{Step} has discovered
that there exists a {\it universal\/} bound $L=L(\Phi)$,
depending on $\Phi$ alone, such that the commutators
$[x,y]$, $x\in G(\Phi,R)$, $y\in E(\Phi,R)$
have elementary width $\le L$ over an {\it arbitrary\/} commutative ring $R$. (Previously in \cite{SiSt} and
\cite{SV} similar results were proven for finite-dimensional
rings, with the bound $L$ depending on $\Phi$ and
dimension $\dim(R)$).
\par
Thus, for all Chevalley groups bounded elementary generation and bounded commutator width are
equivalent! Morally, this says that there are very 
few commutators in $x\in G(\Phi,R)$, not much more
than elementary generators.
\par
But of course the actual bound for commutator width
will be much smaller than the elementary width.
So far, using the results of Smolensky
\cite{Sm} we were able to prove that for a
Dedekind ring of arithmetic type $R$ with the infinite
multiplicative group $R^*$ every element of
$G_{\sic}(\Phi,R)$ is a product of not more than
4,5,6 or 7 commutators, depending on the type $\Phi$
and on whether $R$ is a number ring or a function ring,
this result is contained in \cite{KPV2}.
\par\smallskip
$\bullet$ Above,
Theorem~\ref{result} is stated only for {\it simply laced} root systems. But there is very strong evidence that
suggests that the same is true for all {\it non-symplectic} root systems. We strongly believe in the following
statement and are tempted to call it a conjecture.
\begin{prob}
\label{conj1}
Let $R$ be the Dedekind ring of arithmetic type with infinitely many units and $\Phi$ be a reduced irreducible non-symplectic root system {\rm(}i.e., $\Phi\neq\mathsf A_1,\,\mathsf B_2,\,\mathsf C_l${\rm)}. Then $\mathrm K_2(\Phi,\,R)$ is a finite central subgroup of $\mathrm{St}(\Phi,\,R)$.
\par
In particular, $\mathrm{St}(\Phi,\,R)$ is boundedly generated by the set $X=\{x_\alpha(r)\mid r\in R,\ \alpha\in\Phi\}$.
\end{prob}

\begin{rk}
The centrality of $\mathrm K_2(\Phi,\,R)$ in fact holds for any commutative ring $R$ and any reduced irreducible root system $\Phi$ of rank at least $3$. This result was first proven in~\cite{vdK-another} for $\Phi=\mathsf A_l$, and then in~\cite{L-another,S-amalgams,LS-another,V-unitary,LSV-centrality} for the other root systems. However, if $\Phi$ has rank $2$ then as shown in~\cite{wendt}
centrality may fail even for some very nice rings.
\end{rk}

As Proposition~\ref{cl} shows, one cannot expect an analogue of this to hold in the symplectic case.
However, one can still hope that $\mathrm{St}(\mathsf C_l,\,R)$, $l\ge 3$, is boundedly generated and
could try to approach it by other means.
We state the following problem.

\begin{prob}
Let $R$ be a Dedekind ring of arithmetic type with infinitely many units. Is $\mathrm{St}(\mathsf C_l,\,R)$ boundedly elementarily generated{\rm?}
\end{prob}

\par\smallskip
$\bullet$ Yet another aspect is that Theorem B is
much weaker than Theorem A in that the bound
depends on the size of $\mathrm K_2(R)$. The natural question arises, whether there is a uniform bound in this case
too? However, it seems that an answer to this question
is presently out of range, and in any case should involve
some hard core arithmetic.

\noindent
{\it Acknowledgements.} We are grateful to Nikolai Bazhenov, Sergei Gorchinsky, Alexei Myasnikov, Denis Osipov, Ivan Panin, Victor Selivanov, and Dmitry
Timashev for useful discussions of various aspects
of this work and \cite{KPV2}.


\begin{thebibliography}{Gvoz23}

\bibitem[AbMo]{abe-morita} E.\,Abe, J.\,Morita, \href{https://doi.org/10.1016/0021-8693(88)90272-4}{``Some Tits Systems with Affine Weyl Groups in Chevalley Groups} \href{https://doi.org/10.1016/0021-8693(88)90272-4}{over Dedekind Domains''}, {\it J. Algebra} {\bf115} (1988) 450--465.


\bibitem[BMS]{BMS} H.\,Bass, J.\,Milnor, J-P.\,Serre, \href{https://doi.org/10.1007/BF02684586}{``Solution of the congruence subgroup problem for $\mathrm{SL}_n$} \href{https://doi.org/10.1007/BF02684586}{($n\geq3$) and $\mathrm{Sp}_{2n}$ ($n\geq2$)''}, {\it Publications math\'ematiques de l’I.H.\'E.S.} {\bf33} (1967) 59--137.

\bibitem[BaTa]{BT} H.\,Bass, J.\,Tate, \href{https://doi.org/10.1007/BFB0073733}{``Milnor ring of a global field''}, in: H.\,Bass (ed), ```Classical' Algebraic K-Theory, and Connections with Arithmetic'', {\it Lecture Notes in Mathematics} {\bf342}, Springer Berlin, Heidelberg (1973).

\bibitem[B]{Bour} N.\,Bourbaki, \href{https://link.springer.com/book/9783540691716}{``Lie groups and Lie algebras: Chapters~4--6''}, Springer Berlin Heidelberg (2002) 1--312.

\bibitem[CaKe]{CaKe}
D.\,Carter, G.\,Keller,
\href{https://doi.org/10.2307/2374319}{``Bounded elementary generation of $\SL_n({\mathcal O})$''}, {\it Amer. J. Math.} {\bf 105} (1983) 673--687.

\bibitem[CKP]{CKP}
D.~Carter, G.~E.~Keller, E.~Paige,
{\it Bounded expressions in\/ $\SL(2,{\mathcal O})$},
preprint Univ. Virginia, 1985, 1--21.

\bibitem[C]{Carter}
R.\,Carter, {``Simple groups of Lie type''}, { Wiley}, London et al. (1972).

\bibitem[Gar]{garland} H.\,Garland \href{https://doi.org/10.2307/1970769}{``A Finiteness Theorem for $\mathrm K_2$ of a Number Field''}, {\it Annals of Mathematics} {\bf94}:3 (1971) 534--548.


\bibitem[Gv23]{Gvoz23} P.\,Gvozdevsky, 
\href{https://doi.org/10.1080/00927872.2022.2139953}{``Width of $\SL(n,O_S,I)$''}, {\it Commun. Algebra} {\bf 51}:4, 1581--1593 (2023).





\bibitem[Har]{harder} G.\,Harder, \href{https://doi.org/10.1007/BF01389786}{``Die Kohomologie $S$-arithmetischer Gruppen \"uber Funktionenk\"orpern''}, {\it Invent. Math.} {\bf42} (1977) 135--175.

\bibitem[HSVZ]{HSVZ}
R.\,Hazrat, A.\,Stepanov, N.\,Vavilov, Z.\,Zhang, \href{http://dx.doi.org/10.1285/i15900932v33n1p139}{``Commutator width
in Chevalley groups''}, {\it Note Mat.} {\bf 33} (2013) 139--170.

\bibitem[Hur]{hurrelbrink} J.\,Hurrelbrink, \href{https://doi.org/10.1007/BF01351716}{``Endlich pr\"asentierte arithmetische Gruppen und $\mathrm K_2$ \"uber Laurent-} \href{https://doi.org/10.1007/BF01351716}{Polynomringen''}, {\it Math. Ann.} {\bf225}:2 (1977) 123--129.

\bibitem[JoZa]{JZ}
B.\,Jordan, Y.\,Zaytman,
\href{https://doi.org/10.1073/pnas.1907728116}{``On the bounded generation of arithmetic $\SL_2$''},
{|it Proc. Natl. Acad. Sci. USA} {\bf 116} (2019) 18880--18882.

\bibitem[Ka77]{vdK-another} W.\,van der Kallen, \href{https://doi.org/10.1016/1385-7258(77)90026-9}{``Another presentation for Steinberg groups''}, {\it Indag.~Math.} {\bf39}:4 (1977) 304--312.

\bibitem[Ka81]{vdKded} W.\,van der Kallen, \href{https://doi.org/10.1007/BFb0089523}{``Stability for $\mathrm K_2$ of Dedekind rings of arithmetic type''}, in: E.\,Friedlander, M.\,Stein (eds), ``Algebraic K-Theory Evanston 1980'', {\it Lecture Notes in Mathematics} {\bf854}, Springer Berlin, Heidelberg (1981).

\bibitem[KMR]{KMR} B.\,Kunyavski\u\i , D.\, W.\, Morris, A.\, S.\, Rapinchuk,
{``Bounded elementary generation of $\SL_2$: approaching the end''},
{\it in preparation}.

\bibitem[KPV]{KPV} B.\,Kunyavski\u\i, E.\,Plotkin, N.\,Vavilov, \href{https://doi.org/10.1007/s40879-023-00627-y}{``Bounded generation and commutator width of} \href{https://doi.org/10.1007/s40879-023-00627-y}{Chevalley and Kac--Moody groups: function case''}, {\it European J. Math.} {\bf 9}:53, (2023), 1--64.

\bibitem[KPV2]{KPV2} B.\,Kunyavski\u\i, E.\,Plotkin, N.\,Vavilov, ``Uniform bounded generation of Chevalley over Dedekind rings of arithmetic type'',
 {\it in preparation}.

\bibitem[Lav]{L-another} A.\,Lavrenov, \href{https://doi.org/10.1016/j.jpaa.2014.12.021}{``Another presentation for symplectic Steinberg groups''}, {\it J.~Pure Appl. Alg.} {\bf219}:9 (2015) 3755--3780.

\bibitem[LS17]{LS-another} A.\,Lavrenov, S.\,Sinchuk, \href{https://doi.org/10.1016/j.jpaa.2016.09.004}{``On centrality of even orthogonal $\mathrm{K}_2$''}, {\it J.~Pure Appl. Alg.} {\bf221}:5 (2017) 1134--1145.

\bibitem[LS20]{LS} A.\,Lavrenov, S.\,Sinchuk, \href{https://doi.org/10.4171/dm/762}{``A Horrocks-Type Theorem for Even Orthogonal $\mathrm K_2$''}, {\it Doc. Math.} {\bf25} (2020) 767--809.

\bibitem[LSV]{LSV-centrality} A.\,Lavrenov, S.\,Sinchuk, E.\,Voronetsky, \href{https://arxiv.org/abs/2009.03999}{``Centrality of $\mathrm K_2$ for Chevalley groups: a pro-} \href{https://arxiv.org/abs/2009.03999}{group approach''}, {\it Israel J. Math. } (2023), {\it in press}.

\bibitem[Mat]{mats} H.\,Matsumoto, \href{https://doi.org/10.24033/asens.1174}{``Sur les sous-groupes arithm\'etiques des groupes semi-simples d\'eploy\'es''}, {\it Ann. Sci. Ec. Norm. Sup.} {\bf2} (1969) 1--62.

\bibitem[M]{milnor} J.\,Milnor, \href{https://doi.org/10.1515/9781400881796}{``Introduction to Algebraic K-theory''}, {\it Annals of Mathematics Studies} {\bf72}, Princeton University Press (1972).

\bibitem[MH]{milnor-husemoller} J.\,Milnor, D.\,Husemoller, \href{https://doi.org/10.1007/978-3-642-88330-9}{``Symmetric bilinear forms''}, Springer Berlin Heidelberg (1973).

\bibitem[MRS]{MRS}  A.\,Morgan, A.\,Rapinchuk, B.\,Sury. \href{https://doi.org/10.2140/ant.2018.12.1949}{``Bounded generation of $\mathrm{SL}_2$ over rings of $S$-integers} \href{https://doi.org/10.2140/ant.2018.12.1949}{with infinitely many units''}, {\it Algebra Number Theory} {\bf12}:8 (2018) 1949--1974.

\bibitem[MoRe]{morita-rehmann} J.\,Morita, U.\,Rehmann, \href{https://doi.org/10.1007/BF02571324}{``Symplectic $\mathrm K_2$ of Laurent polynomials, associated} \href{https://doi.org/10.1007/BF02571324}{Kac–Moody groups and Witt rings''}, {\it Math. Z.} {\bf206}:1 (1991) 57--66.

\bibitem[Mor]{Mor}
D.\,Morris, \href{http://eudml.org/doc/128899}{``Bounded generation of $\mathrm{SL}(n,A)$ {\rm(}after
D.\,Carter,
G.\,Keller, and E.\,Paige{\rm)}''}, {\em {New York J. Math.}} {\bf 13} (2007)
383--421.

\bibitem[Pan]{panin} I.\,Panin, \href{http://mi.mathnet.ru/aa27}{``Stabilization for orthogonal and symplectic algebraic K-theories''}, {\it Algebra i Analiz} {\bf1}:3 (1989) 172-195; {\it Leningrad Math. J.} {\bf1}:3 (1990) 741--764.

\bibitem[Pl91]{Pl1}
E.\,Plotkin, \href{https://doi.org/10.1007/BF02988480}
{``Surjective stabilization for $\K_1$-functor for some exceptional Chevalley} \href{https://doi.org/10.1007/BF02988480}{groups''},
{\it Zap. Nauch. Sem. LOMI} {\bf 198} (1991), 65--88;
English transl. in {\it J. Soviet Math.} {\bf 64}:1 (1993), 751--767.

\bibitem[Pl98]{Pl2}
E.\,Plotkin, \href{https://doi.org/10.1006/jabr.1998.7535}
{``On the stability of the \/$\K_1$-functor for Chevalley groups of type\/ ${\mathrm{E}}_7$''},
{\it J.~Algebra} {\bf 210}:1 (1998), 67--85.

\bibitem[Que]{queen} C.\,Queen, \href{https://doi.org/10.1007/BF01229702}{``Some arithmetic properties of subrings of function fields over finite fields''}, {\it Arch. Math} {\bf26} (1975) 51--56.

\bibitem[Reh]{rehmann} U.\,Rehmann, \href{https://www.math.uni-bielefeld.de/~rehmann/Publ/praesentation.pdf}{``Pr\"asentationen von Chevalleygruppen \"uber $k[t]$''}, {\it preprint} (1975), {\tt https://www.math.uni-bielefeld.de/$\sim$rehmann/Publ/praesentation.pdf}.

\bibitem[Sch]{sch} M.\,Schlichting, \href{https://doi.org/10.1016/j.crma.2019.08.001}{``Symplectic and orthogonal K-groups of the integers''}, {\it Comptes Rendus Mathematique} {\bf357}:8 (2019) 686--690.

\bibitem[Sin]{S-amalgams} S.\,Sinchuk, \href{https://doi.org/10.1016/j.jpaa.2015.08.003}{``On centrality of $\mathrm{K}_2$ for Chevalley groups of type $\mathsf E_l$''}, {\it J.~Pure Appl. Alg.} {\bf220}:2 (2016)
857--875.

\bibitem[SiSm]{SiSm}
S.\,Sinchuk,  A.\,Smolensky,
\href{https://doi.org/10.1142/S0218196718500418}{ ``Decompositions of congruence subgroups of Chevalley} 
\href{https://doi.org/10.1142/S0218196718500418}{groups''},
{\it Intern. J. Algebra Comput.} {\bf 28} (2018) 935--958.

\bibitem[SiSt]{SiSt}
A.\,Sivatski, A.\,Stepanov,
\href{http://dx.doi.org/10.1023/A:1007730801851}{``On the word length of commutators in\/
$\mathrm{GL}_n(R)$''},
{\it $\K$-Theory} \textbf{17} (1999) 295--302.

\bibitem[Smo]{Sm}
A.\,Smolensky,
\href{https://doi.org/10.1515/jgth-2018-0035}{ ``Commutator width of Chevalley groups over rings of stable rank $1$''},
J. Group Theory  {\bf 22} (2019), 83--101.

\bibitem[SSV]{SSV}
A.\,Smolensky, B.\,Sury, N.\,Vavilov, \href{htpps://doi.org/10.22108/IJGT.2012.467}
{``Gauss decomposition for Chevalley groups, revis-} \href{htpps://doi.org/10.22108/IJGT.2012.467}
{ited''},
{\it Intern. J. Group Theory} {\bf 1}:1  (2011), 3--16.

\bibitem[Sta]{stavrova} A.\,Stavrova, \href{https://doi.org/10.1515/jgth-2019-0100}{``Chevalley groups of polynomial rings over Dedekind domains''}, {\it Journal of Group Theory} {\bf23}:1 (2020) 121--132.

\bibitem[St71]{Sgrc} M.\,Stein, \href{https://doi.org/10.2307/2373742}{``Generators, relations and coverings of Chevalley groups over commutative} \href{https://doi.org/10.2307/2373742}{rings''}, {\it Amer. J. Math.} {\bf93}:4 (1971) 965--1004.

\bibitem[St73]{Sdim0} M.\,Stein, \href{https://doi.org/10.1090/S0002-9947-1973-0327925-8}{``Surjective stability in dimension $0$ for $\mathrm K_2$ and related functors''}, {\it Trans. Amer. Math. Soc.} {\bf178} (1973) 165--191.

\bibitem[St78]{stein} M.\,Stein, \href{https://doi.org/10.4099/math1924.4.77}{``Stability theorems for $\mathrm K_1$, $\mathrm K_2$ and related functors modeled on Chevalley} \href{https://doi.org/10.4099/math1924.4.77}{groups''}, {\it Japan. J. Math} {\bf4}:1 (1978) 77--108.

\bibitem[S]{Steinberg}
R.\,Steinberg,
Lectures on Chevalley Groups,
University Lecture Series, vol.~66, Amer. Math. Soc., Providence, RI, 2016.

\bibitem[Ste]{Step}
A.\,Stepanov,
\href{https://doi.org/10.1016/j.jalgebra.2015.11.031}{``Structure of Chevalley groups over rings via universal localization''},
{\it J.~Algebra} \textbf{450} (2016) 522--548.

\bibitem[StVa]{SV}
A.\,Stepanov, N.\,Vavilov,
\href{https://doi.org/10.1007/s11856-011-0109-2}{``On the length of commutators in Chevalley groups''},
{\it Israel J. Math.} {\bf 185} (2011), 253--276.

\bibitem[Sus]{suslin} A.\,Suslin, \href{https://doi.org/10.1007/BF00533985}{``Torsion in $\mathrm K_2$ of fields''}, {\it K-Theory} {\bf1} (1987) 5--29.

\bibitem[SuTu]{ST} A.\,Suslin, M.\,Tulenbaev, \href{https://dx.doi.org/10.1007/BF01091768}{``Stabilization theorem for the Milnor $\mathrm K_2$-functor''}, {\it Journal of Soviet Mathematics} {\bf17}:2 (1981) 1804--1819.

\bibitem[Tav]{Tavgen}
O.\,Tavgen, \href{https://doi.org/10.1070/IM1991v036n01ABEH001950}
{``Bounded generation of Chevalley groups over rings of algebraic $S$-integers''}.
{\it Izv. Akad. Nauk SSSR Ser. Mat.} {\bf 54}:1 (1990), 97-–122;
English transl. in {\it Math. USSR Izv.}  {\bf 36}:1 (1991), 101--128.

\bibitem[Tr21]{trost} A.\,Trost, \href{https://arxiv.org/abs/2108.12254}{``Bounded generation by root elements for Chevalley groups defined over} \href{https://arxiv.org/abs/2108.12254}{rings of integers of function fields with an application in strong boundedness''}, {\tt {arxiv:2108.12254}}.

\bibitem[Tr22]{trost2} A.\,Trost, \href{https://arxiv.org/abs/2206.13958}{``Elementary bounded generation for $\mathrm{SL}_n$ for global function fields and $n\geq3$''}, {\tt {arXiv:2206.13958}}.

\bibitem[Va91]{Vav91}
N.\,Vavilov,
{\it Structure of Chevalley groups over commutative rings},
in: `Proc.\ Conf.\ Nonassociative Algebras
and Related Topics (Hiroshima, 1990),
World Sci.\ Publ., London et al., 1991, pp.~219--335.

\bibitem[Va01]{Vav01}
N.\,Vavilov, \href{https://doi.org/10.1023/B:JOTH.0000017882.04464.97}{``Do it yourself structure constants for Lie algebras
of types $\mathsf E_l$''},
{\it J. Math. Sci.}, {\bf 120}:4 (2004) 1513--1548.

\bibitem[Va08]{Vav08}
N.\,Vavilov,
\href{http://dx.doi.org/10.1090/S1061-0022-08-01008-X}{``Can one see the signs of structure constants?}
{\it St. Petersburg Math. J.} {\bf 19}:4 (2008) 519--543

\bibitem[VaPl]{VP} N.\,Vavilov, E.\,Plotkin, \href{https://doi.org/10.1007/BF00047884}{``Chevalley groups over commutative rings: I. Elementary} 
\href{https://doi.org/10.1007/BF00047884}{calculations''}, {\it Acta Applicandae Mathematica} {\bf45} (1996) 73--113.

\bibitem[VSS]{VSS}
N.\,Vavilov, A.\,Smolensky, B.\,Sury, \href{https://doi.org/10.1007/s10958-012-0826-z}
{``Unitriangular factorizations of Chevalley groups''},
{\it Zap. Nauchn. Semin. POMI} {\bf 388} (2011), 17--47;
English transl. in
{\it J. Math. Sci., New York} {\bf 183} (2012), 584--599.

\bibitem[Vor]{V-unitary} E.\,Voronetsky, \href{https://arxiv.org/abs/2005.02926}{``Centrality of odd unitary $\mathrm K_2$-functor''}, {\tt{arXiv:2005.02926}}.

\bibitem[Vse]{Vs}
M.\,Vsemirnov,
\href{https://doi.org/10.1093/qmath/has044}{``Short unitriangular factorizations of
$\SL_2({\Bbb Z}[1/p])$''},
{\it Q. J. Math.} {\bf 65} (2014) 279--290.


\bibitem[W]{Kbook} C.\,Weibel, \href{https://doi.org/10.1090/gsm/145}{``The $\mathrm K$-book: An introduction to algebraic $\mathrm K$-theory''}, {\it Amer. Math. Soc.}, Grad. Texts Math. {\bf145} (2013).

\bibitem[Wen]{wendt} M.\,Wendt, \href{https://doi.org/10.1016/j.jpaa.2012.03.004}{``On homotopy invariance for homology of rank two groups''}, {\it J.~Pure Appl. Alg.} {\bf216}:10 (2012) 2291--2301.

\bibitem[Zak]{Zak} K.\,Zakiryanov,
\href{https://doi.org/10.1007/BF01978846}{``Symplectic groups over rings of algebraic integers have finite
width} \href{https://doi.org/10.1007/BF01978846}{over the elementary matrices''}, {\it Algebra i Logika} {\bf 24}:6 (1985) 667--673;
English transl. in {\it Algebra and Logic} {\bf 24}:6, 436--440.
\end{thebibliography}
\end{document}